\documentclass[11pt]{amsart}

\usepackage{amsthm}
\usepackage{amssymb}
\usepackage{color}
\usepackage{amsmath,amscd}
\usepackage{epsfig}
\usepackage{stmaryrd, mathrsfs, enumerate, dsfont, bbm}
\usepackage{mathbbol}

\theoremstyle{plain}
\newtheorem{theorem}{Theorem}[section]
\newtheorem{lemma}[theorem]{Lemma}

\newtheorem{proposition}[theorem]{Proposition}

\theoremstyle{definition}

\newtheorem{example}[theorem]{Example}
\newtheorem{definition}[theorem]{Definition}

\newcommand{\RR}{\mathbb{R}}

\newcommand{\del}{{\partial}}
\newcommand{\dbar}{{\overline{\delta}}}
\newcommand{\ebar}{{\overline{\epsilon}}}

\newcommand{\GH}{{\text{GH}}}

\newcommand{\F}{\mathcal{F}}
\newcommand{\R}{\mathcal{R}}

\newcommand{\rstr}{\:\mbox{\rule{0.1ex}{1.2ex}\rule{1.1ex}{0.1ex}}\:}

\newcommand{\dF}{d_{\mathcal{F}}}
\newcommand{\dGH}{d_{GH}}

\newcommand{\be}{\begin{equation}}
\newcommand{\ee}{\end{equation}}

\DeclareMathOperator{\Ric}{Ric}

\DeclareMathOperator{\Vol}{vol}
\DeclareMathOperator{\diam}{diam}
\DeclareMathOperator{\spt}{spt}

\DeclareMathOperator{\Lip}{Lip}
\DeclareMathOperator{\0}{\mathbf{0}}
\DeclareMathOperator{\M}{\mathbf{M}}
\DeclareMathOperator{\I}{\mathbf{I}}

\title[Intrinsic flat vs Gromov-Hausdorff]{Intrinsic flat convergence with bounded Ricci curvature}

\author{Michael Munn}
\address{New York University, Courant Institute, 251 Mercer St, New York, NY 10012}
\email{munn@nyu.edu}

\begin{document}
\maketitle

\begin{abstract}
In this paper we  address the relationship between Gromov-Hausdorff limits and intrinsic flat limits of complete Riemannian manifolds. In \cite{SormaniWenger2010, SormaniWenger2011}, Sormani-Wenger show that for a  sequence of Riemannian manifolds with nonnegative Ricci curvature, a uniform upper bound on diameter, and non-collapsed volume, the intrinsic flat limit exists and agrees with the Gromov-Hausdorff limit. This can be viewed as a non-cancellation theorem showing that for such sequences points don't cancel each other out in the limit. 

In this paper, we extend this result to show that there is no cancellation when replacing the assumption of nonnegative Ricci curvature with a two-sided bound on the Ricci curvature. 

\end{abstract}

\section{Introduction}
\label{section-introduction}

This paper is concerned with the relationship between Gromov-Hausdorff limits and intrinsic flat limits of complete oriented Riemannian manifolds. The Gromov-Hausdorff distance $\dGH$, introduced in \cite{Gromov}, generalizes to pairs of compact metric spaces the Hausdorff distance for subsets of a metric space. In this way, $\dGH$ defines a metric on all compact metric spaces so that for two Riemannian manifolds $M_1, M_2$, $\dGH(M_1, M_2) = 0$ if and only if there exists an isometry between them. The intrinsic flat distance $\dF$, introduced in \cite{SormaniWenger2011}, generalizes the flat distance of Federer-Fleming, defined for integral current spaces in $\mathbb{R}^n$, to pairs of metric integral current spaces. As such, $\dF$ depends on the Riemannian manifolds as oriented metric spaces equipped with an integration form so that $\dF(M_1, M_2) = 0$ if and only if there is an orientation preserving isometry between them. 

Both the Gromov-Hausdorff and flat distance play an important role in metric and Riemannian geometry and geometric measure theory. And, since its introduction, the intrinsic flat convergence has proven very useful in studying properties of both Riemannian manifolds and integral current metric spaces spanning both these fields \cite{LakzianSormani, LeeSormani}. Here, we extend results of Sormani-Wenger \cite{SormaniWenger2010} giving sufficient conditions which guarantee that the limit spaces of these two notions of convergence must agree.

It follows from Gromov's Precompactness Theorem that any sequence of compact Riemannian manifolds $(M_i, g_i)$ with bounded Ricci curvature $\Ric_{M_i} \geq -(n-1)$ and bounded diameter $\diam(M_i) \leq D$ has a subsequence which converges in the Gromov-Hausdorff topology to a compact metric space $(X,d)$. It is important to understand how the Gromov-Hausdroff limits compare with the intrinsic flat limits. To this end, Sormani-Wenger show  

\begin{theorem}
\label{SWThm}
{\em(c.f. \cite[Theorem 4.16]{SormaniWenger2010} and \cite[Theorem 7.1]{SormaniWenger2011})} 
If a sequence of oriented Riemannian manifolds $(M_i, g_i)$ satisfies the bounds
\[
\Ric_{M_i} \geq 0,  \quad \diam(M_i) \leq D, \quad  \Vol(M_i) \geq v_0>0,
\]
then the Gromov-Hausdorff limit and the intrinsic flat limit agree; i.e., if the sequence $(M_i, g_i)$ converges in the Gromov-Hausdorff sense to a metric space $(X,d)$, then it converges in the intrinsic flat sense to the integral current space $(X,d,T)$, where $T$ denotes the $n$-dimensional integral current structure on $X$, and the support of $T$ coincides with the Gromov-Hausdorff limit $X$.
\end{theorem}

From the properties of the intrinsic flat convergence, in general,  $\spt(T) \subset X$ and the inclusion may be strict; i.e. the limit current space is supported in a strict subset of the Hausdorff limit. This is because intrinsic flat limits do not include points with zero density. This phenomena is demonstrated in Example \ref{Main Ex} (c.f. Examples 6.1-6.3 in \cite{SormaniWenger2011}). With these examples in mind, the case of strict inclusion is often referred to as `cancellation'. In this sense, Theorem \ref{SWThm} can be viewed as a non-cancellation statement for intrinsic flat limits. That is to say, points in the Gromov-Hausdorff limit are not `lost' when viewed in the intrinsic flat limit. 

In addition, it follows from Theorem \ref{SWThm} that the limit space $X$ is $\mathcal{H}^n$ rectifiable. Thus, this non-cancellation theorem and the  intrinsic flat convergence provide a method of extending the rectifiability properties achieved by Cheeger-Colding \cite{ChCoI} to a much larger class of manifolds. 

In this paper, we prove a similar result for manifolds with bounded Ricci curvature. We show
\begin{theorem}
\label{MainThm}
For $H \in \mathbb{R}$, if a sequence of oriented Riemannian manifolds $(M_i, g_i)$ satisfies the bounds
\[
|\Ric_{M_i}| \leq (n-1)H,  \quad \diam(M_i) \leq D, \quad  \Vol(M_i) \geq v_0>0,
\]
then the Gromov-Hausdorff limit and the intrinsic flat limit agree.
\end{theorem}

In \cite{LakzianSormani} it is conjectured that the bound on the Ricci curvature in Theorem \ref{SWThm} holds more generally, assuming only a uniform lower bound on the Ricci curvature. At the current time, it is not clear to us how to remove the upper bound on $\Ric$ in Theorem \ref{MainThm}.

The original argument of Sormani-Wenger relies on bounding from below the absolute filling radius of slices of distance spheres in the integral current space (see Section 4 of \cite{SormaniWenger2010}). The notion of filling radius was introduced by Gromov in \cite{Gromov} to obtain global volume estimates in the class of manifolds with the same local contractibility functions. Recall, that a lower bound on Ricci curvature naturally imposes a geometric contractibility function on a smooth Riemannian manifold \cite{Perelman}. In \cite{SormaniWenger2011}, Sormani-Wenger generalize to integral current spaces an earlier result of Greene-Petersen \cite{GP} for Riemannian manifolds where they prove that the lower bound on the  filling radius on such slices is controlled by a local geometric contractibility function of the manifold. 

Our approach avoids any use of the filling radius or the filling volume of Gromov. Instead, we show how the Gromov-Hausdorff convergence, paired with volume convergence at regular points in the limit, controls the size (measured in the intrinsic flat sense) of distance spheres centered at these points in the limit. 

The remainder of the paper is arranged as follows: In Section \ref{background} we review the necessary background, addressing Gromov-Hausdorff and intrinsic flat distance and the relationship between the two notions of convergence these distances induce. In Section \ref{section-proof} we give a proof of our Main Theorem \ref{MainThm}. 

{\bf Acknowledgements.} We wish to thank Christina Sormani and Xiaochun Rong for helpful comments and discussion. We also thank Aaron Naber for noticing a crucial error in a previous version of this paper.

\section{Background}
\label{background}
In this section we recall various notions of convergence for Riemannian manifolds and the relevant background material and notation we will need going forward.

Let $(M^n, g)$ be a complete oriented Riemannian manifold. For $p\in M$ and $r >0$, denote by $B^M_r(p)$ the geodesic ball of radius $r$ centered at $p$ and let $S^M_r(p) = \partial B^M_r(p)$. Similarly, denote by $B^{\RR^n}_r(0)$ the ball of radius $r$ in $\RR^n$ centered at the origin. Let  $g_{\RR^n}$ denote the standard Euclidean metric on $\mathbb{R}^n$ and $\omega_n$ to denote the volume of the unit ball $B^{\RR^n}_0(1) \subset \mathbb{R}^n$. We use $\Vol_{M}$ to denote the Riemannian volume on $(M,g)$ and $\mathcal{H}^k$ to denote the $k$-dimensional Hausdorff measure.

\subsection{Gromov-Hausdorff Distance.} 
\label{subsection-GH}
The Gromov-Hausdorff distance, introduced in 1981 in \cite{Gromov}, is an extension of the Hausdorff distance between subsets of a single metric space. Let $(X,d)$ be a metric space and $A,B \subset X$, the Hausdorff distance is given by
\[
d_H\left(A, B\right) = \inf\{r ~:~ A \subset T(B,r) \text{ and } B \subset T(A,r)\},
\]
where $T(B,r) = \{x \in X ~:~ d(x,B) < r\}$ and $T(A,r)$ is defined similarly.
The Gromov-Hausdroff distance measures the infimum of the Hausdorff distance between all possible isometric embeddings
of these spaces into a common metric space. 
\begin{definition}
Given a pair of metric spaces $(X,d_X)$ and $(Y,d_Y)$ the {\em Gromov-Hausdorff distance} between them is
\[
\dGH(X,Y) = \inf_{\substack{\varphi_X:X \to Z\\\varphi_Y:Y\to Z}} \{d_H^Z(\varphi_X(X), \varphi_Y(Y))\},
\]
where the infimum is taken over all common metric spaces $Z$ and all isometric embeddings $\varphi_X:X \to Z$ and $\varphi_Y:Y \to Z$.
\end{definition}
It follows that if $X$ and $Y$ are two compact metric spaces with $\dGH(X,Y)=0$ then $X$ and $Y$ are isometric. In this way, by considering equivalence classes of isometric spaces, $\dGH$ defines a metric on the collection of all compact metric spaces. If a metric space $X$ is precompact, it is necessary to take the metric completion $\overline{X}$ before computing the Gromov-Hausdorff distance.
\begin{definition}
Given a precompact metric space $X$, the \textit{metric completion $\overline{X}$ of $X$} is the space of Cauchy sequences $\{x_j\}$ in $X$ with the metric
\[
d\left(\{x_j\}, \{y_j\} \right) = \lim_{j \to \infty} d_X(x_j, y_j)
\]
and where two Cauchy sequence are identified if the distance between them is 0. 
\end{definition}

The Gromov-Hausdorff distance can also be defined for noncompact metric spaces. In this case, one considers the {\em pointed Gromov-Hausdorff} distance 
\[
\dGH\left((X,d_X, x), (Y,d_Y,y) \right) = \inf\{d_H^Z(\varphi_X(X), \varphi_Y(Y)) + d_Z(\varphi_X(x),\varphi_Y(y))\},
\]
again taking the infimum over all isometric embeddings into $(Z,d_Z)$. A sequence of locally compact metric spaces $(X_i, d_{X_i}, x_i)$ converges in the pointed Gromov-Hausdorff sense to $(X,d_X,x)$ if for all $R>0$,
\[
\left(\overline{B_{x_i}(R)}\subset X_i , d_{X_i}\right) \xrightarrow{~~\dGH~~} \left(\overline{B_{x}(R)}\subset X , d_{X}\right).
\]
Lastly, we mention {\em measured Gromov-Hausdorff} convergence. Given a sequence of metric spaces $(X_i, d_{X_i})$ equipped with Borel probability measures $\nu_{X_i}$, we say the sequence converges in the measured Gromov-Hausdorff sense, denoted 
\[
\left(X_i, d_{X_i}, \nu_{X_i}\right) \xrightarrow{~~m\dGH~~} \left(X, d_{X}, \nu_{X}\right),
\]
provided $(X_i, d_{X_i}) \xrightarrow{~~\dGH~~} (X,d_X)$ and there is a sequence of Borel maps $f_i: X_i \to X$ such that $\displaystyle{\lim_{i\to \infty} \left({f_i}\right)_*\nu_{X_i} = \nu_X}$ in the weak-$\ast$ topology.

\subsection{Regular Sets and Renormalized Limit Measures} In \cite{ChCoI}, Cheeger-Colding initiate a detailed study of the structure of the limit spaces which arise as the pointed Gromov-Hausdorff limits of sequences of complete Riemannian manifold with Ricci curvature bounded below. Let $(X,d_X, x)$ be the pointed Gromov-Hausdroff limit of a sequence of pointed Riemannian manifolds $\{(M^n_i, p_i)\}$ with $\Ric_{M^n_i} \geq -(n-1)H$. 

The natural measure associated to elements in the sequence is the Riemannian volume measure $\Vol_{M_i}$. By taking renormalizations of these measures along a suitable subsequence, it is possible to equip the limit space $(X,d_X)$ with a limit measure $\nu$ (see also \cite{Fukaya}) so that 
\[
(M^n_i, d_{g_i}, \Vol_{M_i}) \xrightarrow{~~m\dGH~~} \left(X, d_{X}, \nu_{X}\right).
\]
Given a ball $B_x(r) \subset X$, take (along an appropriate subsequence) 
\[
\nu(B_x(r))  = \lim_{i \to \infty} \frac{1}{\Vol_{M_i}(B_{p_i}(1))} \Vol_{M_i}(B_r(p_i)).
\]

In this paper, we are concerned with sequences which have volume non-collapsed. In this situation, the limit measure exists without the necessity of passing to a subsequence and any such limit measure $\nu$ defined this way is a multiple of the $n$-dimensional Hausdorff measure $\mathcal{H}^n$ on $X$. 

A tangent cone at $x \in X$ is obtained by `blowing up' the metric $d_X$ near $x$. More precisely, a {\em tangent cone at $x \in X$} is a complete pointed Gromov-Hausdorff limit $(X_x, d_{\infty}, x_{\infty})$ of a sequence of rescaled spaces $(X, \epsilon_i^{-1} d_X, x)$ where the $\{\epsilon_i\}$ is a positive sequence with $\epsilon_i \to 0$ and $d_{\infty}$ denotes the distance metric of the resulting limit space $X_x$. By Gromov's compactness theorem, tangent cones exist for all $x\in X$ but might depend on the choice of the convergent (sub)sequence. 

\begin{definition}(c.f. \cite[Definition 0.1]{ChCoI})  A point $x \in X$ is called {\em regular} if, for some $k$, every tangent cone is isometric to $\mathbb{R}^k$. Denote the set of $k$-regular points by $\mathcal{R}^k$ and set $\mathcal{R} = \cup_k \mathcal{R}^k$. Call $\mathcal{R}$ the {\em regular set}. 

A point is called a {\em singular point} if it is not regular. Denote by $\mathcal{S}$ the singular set of $Y$.
\end{definition}

\begin{theorem}
\label{S meas}
{\em (c.f. \cite[Theorem 2.1]{ChCoI})}
For any renormalized limit measure, $\nu(\mathcal{S}) = 0$. Thus, $\nu(\mathcal{R})$ is dense in $X$.
\end{theorem}

\subsection{Intrinsic Flat Distance}
\label{subsection-IF}
The intrinsic flat distance was introduced to better understand the limits of sequences of Riemannian manifolds. Much like the Gromov-Hausdroff distance, it is also defined on a larger class of metric spaces, namely metric integral current spaces. In \cite{AK}, Ambrosio-Kirchheim extend the classical Federer-Fleming theory of integral currents in Euclidean space to integral currents in general metric spaces. Using this theory of metric integral currents, we define the intrinsic flat distance between metric integral current spaces. 

An integral current space is denoted $(X,d,T)$ where $(X,d)$ is a metric space and $T$ is an integral current structure supported on $X$. For $m \geq 0$, the space of integral $m$-currents in $X$, denoted by $\mathbf{I}_m(X)$, are functionals on generalized $m$-forms. The integral current provides both an orientation and a measure $||T||$ called the mass measure of $T$. 

Any compact, oriented Riemannian manifold $(M^n,g)$ with or without boundary can be viewed as a metric space $(M,d_g)$ possessing an integral current $T \in \mathbf{I}_n(M)$, defined as integration over $M$, i.e. for any $n$-dimensional differential form $\omega \in \Omega^n(M)$ on $M$, let 
\[
T(\omega) = \int_{M} \omega,
\]
while the mass measure is just the Riemannian volume measure on $M$ and, assuming multiplicity one, the mass is simply the volume of the manifold.

More precisely, let $(Z,d)$ be a metric space and let $\mathcal{D}^m(Z)$ be the set of $(m+1)$-tuples $(f, \pi_1, \dots, \pi_m)$ of Lipschitz functions on $Z$ with $f$ bounded. 
\begin{definition}Z
An {\em $m$-dimensional metric current $T$ on $Z$} denoted $\mathbf{M}_m(Z)$ is a multi-linear functional on $\mathcal{D}^m(Z)$ which satisfies
\begin{enumerate}[(i)]
\item If $\pi_i^j$ converges pointwise to $\pi_i$ as $j \to \infty$ and $\sup_{i,j} \Lip(\pi_i^j) < \infty$, then 
\[
T(f, \pi_1^j,\dots, \pi_m^j ) \rightarrow T(f, \pi_1, \dots \pi_m).
\]
\item If $\{z \in Z ~:~ f(x) \neq 0\}$ is contained in the union $\displaystyle{\cup_{i=1}^m B_i}$ of Borel sets $B_i$ and if $\pi_i$ is constant on $B_i$, then 
\[
T(f, \pi_1, \dots, \pi_m) =0.
\]
\item There edits a finite Borel measure $\mu$ on $Z$ such that 
\[
|T(f, \pi_1, \dots, \pi_m)| \leq \prod_{i=1}^m \Lip(\pi_i) \int_Z |f| d\mu.
\]
The minimal such Borel measure $\mu$ is called the {\em mass measure of $T$}, denoted $||T||$. 

\end{enumerate}
\end{definition}
The mass of a current $T  \in \mathbf{I}_n(Z)$ is a kind of weighted volume of the space $Z$ depending on the current structure $T$ and defined by
\[
\M(T) : = ||T||(Z) = \int_Z d||T||.
\]   
The $\spt(T)$ is the set of positive lower density for the measure $||T||$
\[
\spt(T) = \{z \in Z ~:~ ||T||(B_x(\epsilon)) >0 \text{  for all  } \epsilon >0\}.
\]
Restricting $T$ to a Borel set $A \subset Z$, define $T\rstr A$ by 
\[
(T \rstr A)(f, \pi_1, \dots, \pi_m) := T(f \chi_A, \pi_1, \dots, \pi_m).
\]
Along these same lines,  taking $m \geq 1$, the boundary of $T$ is defined by
\[
\partial T (f, \pi_1, \dots, \pi_{m-1}) := T(1, f, \pi_1, \dots, \pi_{m-1}).
\]
Furthermore, set $\partial(Z,d,T) = (\spt(\partial T), d, \partial T)$. We assume $\partial T \in \mathbf{M}_{m-1}(Z)$ so that $T$ is a normal current and set
\[\partial T(\varphi) = T(\partial \varphi).
\]
Of course, when $M$ is a Riemannian manifold, then $\partial M$ is just the usual boundary of $M$ and this agrees with Stoke's Theorem.

The {\em push-forward of $T \in \M_n(Z)$} under a Lipschitz map $\varphi$ from $Z$ to another complete metric space $Y$ is defined as 
\[
{\varphi}_\# T(f, \pi_1, \dots, \pi_m) := T(f\circ \varphi, \pi_1\circ \varphi, \dots, \pi_m\circ \varphi),
\]
for $(f, \pi_1, \dots, \pi_m) \in \mathcal{D}^n(Y)$.

As mentioned above, we are concerned with integral current spaces. For $m \geq 1$, a current $T \in \mathbf{M}_m(Z)$ with $\partial T \in \mathbf{M}_{m-1}(Z)$ is called an integral current, denoted $\mathbf{I}_m(Z)$, provided
\begin{enumerate}[(i)]
\item $||T||$ is concentrated on countably $\mathcal{H}^m$-rectifiable sets and vanishes on $\mathcal{H}^m$-negligible sets
\item For any Lipschitz map $\varphi: Z \to \mathbb{R}^m$ and any open set $U \subset Z$, there exits $\theta \in L^1(\mathbb{R}^m, \mathbb{Z})$ such that $\varphi_\#(T \rstr U) =\displaystyle{\int_{K \subset \mathbb{R}^m} \theta f \det\left(\frac{\partial \pi_i}{\partial x_j} \right) d \mathcal{L}^m}$.
\end{enumerate}

The {\em flat distance} between two integral currents $T_1, T_2 \in \I_n(Z)$ is defined as
\[
d_F\left(T_1, T_2 \right) := \inf \{\M(B^{n+1}) + \M(A^n) ~:~ T_1-T_2 = A + \partial B\}.
\]

Similar in spirit to the Gromov-Hausdorff distance,  given $T_1, T_2 \in \I_n(Z)$ define
\begin{definition}
\label{IFdefn}
{ (Sormani-Wenger \cite{SormaniWenger2011})}
The {\em intrinsic flat distance} between $(X_1, d_1, T_1)$  and $(X_2, d_2, T_2)$ is
\[
\dF \left((X_1, d_1, T_1), (X_2, d_2, T_2) \right) = \inf_{\substack{\\ \varphi_i : X_i \to Z}} \{d_F^Z({\varphi_1}_{\#}T_1, {\varphi_2}_{\#}T_2)\},
\]
where the infimum is taken over all common complete metric spaces $Z$ and all isometric embeddings $\varphi_i: X_i \to Z$.
\end{definition}

Using techniques from geometric measure theory, it is possible to determine precise estimates on the intrinsic flat distance in terms of the Lipschitz distance between the metric spaces, the diameters of the spaces, and their volumes. For our purposes, we make use of the following estimate of Lakzian-Sormani 



\begin{proposition}
\label{lemma-LakzianSormani}
{\bf (c.f. Lemma 4.5 in \cite{LakzianSormani}).}
Let $M_1= (M, g_1)$ and $M_2=(M, g_2)$ be diffeomorphic, oriented pre-compact Riemannian manifolds and suppose that there exists $\epsilon >0$ such that 
\begin{equation}
\label{g1g2bound}
g_1(v,v) < (1+\epsilon)^2 g_2(v,v) \quad \text{ and } \quad g_2(v,v) < (1+\epsilon)^2 g_1(v,v), \quad \forall v \in TM.
\end{equation}
Then for any 
\[a_1 > \frac{\arccos(1+\epsilon)^{-1}}{\pi} \diam(M_2) \qquad \text{ and } \qquad a_2 > \frac{\arccos(1+\epsilon)^{-1}}{\pi} \diam(M_1),\]
there is a pair of isometric embeddings $\varphi_i: M_i \to \overline{M} \times [t_1, t_2]$, for $i=1,2$, which can be used to bound $\GH$-distance and $\mathcal{F}$-distances. Namely,
\be
\dGH(\overline{M_1}, \overline{M_2}) \leq \max(a_1, a_2), \qquad  \text{ and }
\ee
\be \label{LS bound}
\dF(M_1', M_2') \leq \max(a_1, a_2) \left(\Vol^n(M_1) + \Vol^n(M_2) + \Vol^{n-1}(\partial M_1) + \Vol^{n-1}(\partial M_2) \right).
\ee
\end{proposition}

Here $M_i'$ denotes the settled completion of $M_i$. In general, if $(M,g)$ is a Riemannian manifold with singularities on a set $S$ of $\mathcal{H}^n$-measure zero, then the corresponding integral current space is given by the settled completion of $M\setminus S$. 
\begin{definition}
Let $(X,d)$ be a metric space with a measure $\mu$. The {\em settled completion} $X'$ is the collection of points $p $ in the metric completion $\overline{X}$ with have positive lower density
\[
\liminf _{r\to 0} \frac{\mu(B_p(r))}{r^n} > 0.
\]
\end{definition}

\subsection{Relationship between Intrinsic Flat and Gromov-Hausdorff convergence}
There are different ways to measure the convergence of manifolds or metric spaces. The Gromov-Hausdorff and intrinsic flat distance give two notions of convergence for oriented Riemannian manifolds while the Gromov-Hausdorff distance determines a very weak form of convergence. Generally speaking, their limits will not be smooth manifolds themselves and it is important to understand when these two weak forms of convergence agree in their limit. In the simple example below, we see that this is not necessarily always the case (c.f. Examples 6.1-6.3 in \cite{SormaniWenger2011}).
\begin{example}
\label{Main Ex}
Consider the sphere $\mathbb{S}^n$ with a small rounded cylindrical neck attached. As the cylindrical neck collapses, although the diameter and volume remain bounded,  the (negative) Ricci curvature, where the neck joins the sphere, becomes increasingly more negative along the sequence.

In the limit, the Gromov-Hausdorff limit $X_{\infty}$ consists of the sphere $\mathbb{S}^n$ with a straight line segment attached as the remnant of the cylindrical neck. However, in the  intrinsic flat convergence, the orientations of the rounded cylinder cancel each other out in the limit. Thus, the intrinsic flat limit $Y_{\infty}$ is just the sphere $\mathbb{S}^n$. See the Figure 1 below. Here $Y_{\infty} \subsetneq X_{\infty}$. Thus, the uniform lower bound on Ricci in Theorem \ref{MainThm} is necessary.

\begin{figure}[h!]
\label{fig}
\centering
\includegraphics[width=6in, height=2.5in]{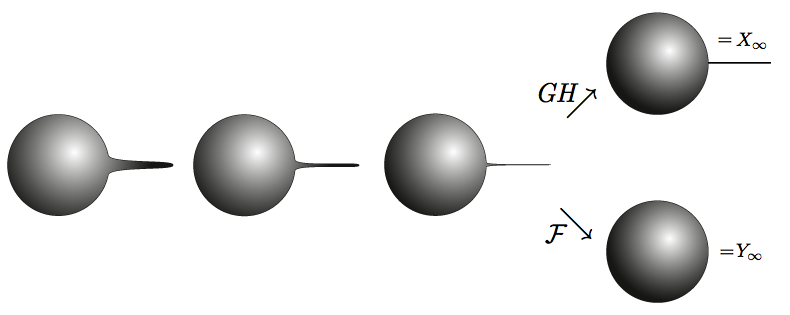}\\
  \caption{Gromov-Hausdorff limit vs Intrinsic Flat limit.}
  \end{figure}

\end{example}

In recent work \cite{Sormani-preprint}, Sormani continues this study of intrinsic flat limits proving two Arzela-Ascoli theorems. One of the essential ingredients of her arguments is the following Lemma which will also be useful for our purposes. Letting $\0$ denote the trivial flat  integral current space $(0,0,0)$, we have 

\begin{lemma}
\label{SormaniPreprintLemma}
{\em \bf (c.f. Lemma 4.1 in \cite{Sormani-preprint}).}
Suppose $M_j \xrightarrow{\F} M_{\infty} = (Y_{\infty}, d_{\infty}, T_{\infty})$ and $p_j \in M_j$ a Cauchy sequence which has no limit in $\overline{Y}_{\infty}$, then 
\begin{align*}
\tag{$\ast$} \label{Fto0condition}
&\text{there exists } \delta >0 \text{ such that, for almost every } r \in (0, \delta), \\
&S^{M_j}_{r}(p_j) \text{ are integral} \text{ current spaces for } j = \{1, 2, \dots\} \text{ and }S^{M_j}_{r}(p_j) \xrightarrow{\F} \0.
\end{align*}
\end{lemma}

This Lemma provides a powerful tool when trying to prove that intrinsic flat limits are not the zero integral current space. With the goal of excluding this behavior of points being `lost' in the intrinsic flat limit, note that if one can find, for $r>0$, a non-flat comparison integral current space $N$ (possibly depending on $r$), such that for all $j$ sufficiently large
\[
\dF(S^{M_j}_r(p_j), N) < \frac{\dF(N, \0)}{2};
\]
then, it follows by the triangle inequality  that
\[
\dF(S^{M_j}_r(p_j), \0) > \frac{\dF(N, \0)}{2} >0,
\]
for all $j$ sufficiently large as well. 
In fact, this observation will be crucial in our argument in the proof of Theorem \ref{MainThm}. See also \cite{SormaniWenger2010, LakzianSormani, LeeSormani} for other means  of estimating the intrinsic flat distance.\\

In the argument to follow, we will often consider metric spaces whose distance has been rescaled. With this in mind, let us set some notation. Consider a metric space $(X,d)$ and $R>0$ a real number. We will denote by $\frac{X}{R}$ the rescaled metric space $(X,\frac{1}{R}d)$. Of course, any Riemannian manifold $(M^n,g)$ can be given an integral current structure induced from the Riemannian distance $d_g$ and orientation form $d\Vol_g$, respectively. Viewed in this way we have the natural relation $(M, \frac{1}{R}d_g) = (M, \frac{1}{R^2} g)$ and we use $M^n/R$ to refer to them both.  

We conclude this section by noting 

\begin{lemma}
\label{scalingLemma} 
{\em \bf (c.f. Lemma 4.4 in \cite{Sormani-preprint}).}
Let $M_i = (X_i, d_i, T_i)$ and let $R>0$. Consider the rescaled integral current spaces given by $M_i' = (X_i, d_i/R, T_i)$. Then,
\[\dF (M_1, M_2) \leq \dF (M_1', M_2') R^n(1+R)\]
\end{lemma}

The idea of the proof relies on the fact (c.f. \cite[Theorem 4.23]{SormaniWenger2011}) that the infimum of Definition \ref{IFdefn} is realized by {\em some} metric space $(Z,d_Z)$ and isometric embeddings $\varphi_i: X_i \to Z$. By rescaling this optimal metric space we can achieve a bound on the intrinsic flat distance for scaled currents.\\

\noindent {\em Proof of Lemma \ref{scalingLemma}.}
Consider $M_i' = (X_i, d_i/R, T_i)$. By Theorem 4.23 in \cite{SormaniWenger2011} there exists some metric space $(Z,d_Z)$ and isometric embeddings which achieve the infimum in the definition of $d_F^Z({\varphi_1}_{\#}T_1, {\varphi_2}_{\#}T_2)$ in the definition of $\dF(M_1', M_2')$. It follows that, for any $x,y \in X_i$,
\[
d_Z(\varphi_i(x), \varphi_i(y)) = \frac{1}{R}d_i(x,y).
\]
Therefore, the metric space  $(Z,Rd_Z)$ is a contender for the infimum in  $\dF(M_1, M_2)$,  since the maps $\varphi_i : (X_i, d_i) \to (Z, Rd_Z)$ are isometric embeddings of $(X_i, d_i)$, by construction. 

Noting how the mass measure $||T||$ and its associated mass $\M$ scale with respect to the metrics $d_Z$ and $R\cdot d_Z$ and taking $\varphi_i : (X_i, d_i) \to (Z, Rd_Z)$ as above, we get (denoting the scaled mass by $\M_R$) 

\begin{align*}
\dF(M_1, M_2) &\leq d_F^Z\left({\varphi_1}_\# T_1, {\varphi_2}_\#T_2 \right)\\
			&\leq \M_R(B^{n+1}) + \M_R(A^n)\\
			& = R^{n+1} \M(B^{n+1}) + R^n \M(A^n)\\
			&\leq \left(\M(B^{n+1}) + \M(A^n)\right)\left(R^{n+1} + R^n \right)\\
			&= \dF (M_1', M_2') R^n(1+R).
\end{align*}
\hfill \qed\\



\section{Proof of Theorem \ref{MainThm}.}
\label{section-proof}
Let $\{(M^n_i, g_i)\}_{i=1}^{\infty}$ be a sequence of closed oriented Riemannian manifolds satisfying
\be\label{conditions}
|\Ric_{M_i}| \leq (n-1)H,  \quad \diam(M_i) \leq D, \quad  \Vol(M_i) \geq v_0>0.
\ee
It follows from Gromov's Precompactness Theorem, that a subsequence $\{M^n_i\}$ converges in the Gromov-Hausdoff sense to a compact metric space $(X_{\infty},d_{{\infty}})$. Recalling Section \ref{subsection-GH}, this means that there exists some compact metric space $Z$ with $X_{\infty} \subset Z$ and a sequence of isometric embeddings $\varphi_i: M^n_i \to Z$ such that  $d^Z_H(\varphi_i(M^n_i), X_{\infty}) \to 0$ as $i \to \infty$.

In the same way, we can view each $(M^n_i, g_i)$ as an integral current space. Denote by $T_i := \Lbrack M_i\Rbrack$ the integral $n$-current in $M^n_i$ induced by integration over $M^n_i$. Thus, $\{(M^n_i, d_{g_i}, T_i)\}$ is a sequence of integral current spaces. Given that there is a subsequence which converges in the Gromov-Hausdorff sense, it follows that a further subsequence (which we also denote by $M^n_i$) converges in the intrinsic flat sense to some $(Y_{\infty},d,T)$ where  $X_{\infty} \subset Y_{\infty}$, $d$ is the metric $d_{\infty}$ restricted to $Y_{\infty} =  \spt(T)$ and $\varphi_{\#}T_i$ converge weakly to $T \in \I_n(Z)$.

\begin{theorem}
With $\{(M^n_i, g_i)\}_{i=1}^{\infty}$ as above. Then, 
\[
\overline{Y_{\infty}} = X_{\infty}.
\]
\end{theorem}

\begin{proof}
By Sormani-Wenger's previous work \cite{SormaniWenger2010}, it is known that $\overline{Y_{\infty}} \subset X_{\infty}$. Thus, we need only verify that  $X_{\infty} \subset \overline{Y_{\infty}}$ as well. We suppose not and derive a contradiction. Let $p_i \in M_i$ be a sequence so that $\varphi_i(p_i) \subset Z$ is Cauchy and $\varphi_i(p_i) \to p_{\infty} \in X_{\infty}$ but which has no limit in $\overline{Y_{\infty}}$. It suffices to show that (\ref{Fto0condition}) in Lemma \ref{SormaniPreprintLemma} fails. Note that by construction $S^{M_i}_r(p_i)$ is an integral current space for any $i \in \mathbb{N}$; thus, to obtain a contradiction, we need only verify \\

{\em Claim.} For any $\dbar >0$, the set of all $r \in (0, \dbar)$ for which $S^{M_i}_r(p_i)$ does not converge to the flat integral current space has positive measure.  \\

{\bf Case 1:} $p_{\infty} \in \mathcal{R}$.\\
Suppose $p_{\infty} \in \mathcal{R} \subset X_{\infty}$, where $\mathcal{R}$ denotes the regular set of $X_{\infty}$ as defined by Cheeger-Colding. Since the sequence is assumed to be volume non-collapsed, it follows from \cite{ChCoI} that, as $r \to 0$, 
\[
\dGH\left(\frac{B^{X_{\infty}}_{r}(p_{\infty})}{r}, B^{\RR^n}_0(1)\right) \longrightarrow 0.
\]
Recall, here $\frac{B^{X_{\infty}}_{r}(p_{\infty})}{r}$ denotes the geodesic ball $B^{X_{\infty}}_{r}(p_{\infty}) \subset X_{\infty}$ with the distance metric rescaled by $1/r$ so that it is a unit ball. Furthermore, since $\varphi_i(p_i) \to p_{\infty}$, we have, for a given $r >0$, that  $\dGH\left(\dfrac{B^{M_i}_{r}(p_i)}{r}, \dfrac{B^{X_{\infty}}_{r}(p_{\infty})}{r} \right) \longrightarrow 0$, as $i \to \infty$. Combining these two facts, it follows that for any $\epsilon >0$, there exists $N_1:=N_1(\epsilon,n) >0$ such that, if $i > N_1$ and $0 < r< \frac{1}{N_1}$, then 
\be\label{ineqI}
\dGH \left(\frac{B^{M_i}_r(p_i)}{r}, B^{\RR^n}_0(1)  \right) < \epsilon.
\ee

By a consequence of Colding's Volume Convergence theorem (see \cite[Corollary 2.19]{Colding1997}), under this Gromov-Hausdroff convergence the volume of $\frac{B^{M_i}_r(p_i)}{r}$ and $B^{\RR^n}_1(0)$ can be taken arbitrarily close as well, given $i$ sufficiently large. More precisely, given any $\epsilon >0$, set $\delta_{\epsilon, n, H}:= \delta(\epsilon, n, H) >0$ as given in Colding's theorem and take $N_1(\delta_{\epsilon, n, H})>0$. It follows that,  for $i > N_1(\delta_{\epsilon, n, H})$ and $r < N_1(\delta_{\epsilon, n, H})^{-1}$, 
\be\label{vol close}
|\Vol \left(\frac{B^{M_i}_r(p_i)}{r}\right) - \Vol(B^{\RR^n}_1(0)) | < \epsilon.
\ee

Now it follows as a consequence of the proof of Anderson's $C^{1, \alpha}$-compactness theorem  for compact Riemannian manifolds with bounded Ricci curvature (c.f Theorem 3.2 and Remark 3.3 of \cite{Anderson1990a}) that, since the volume of unit balls in $M^n_i$ converges to the volume of the unit ball in $\RR^n$, a subsequence of the Riemannian metrics converge in the $C^{1,\alpha}$-topology on balls of smaller radius. That is to say, if we denote by $\epsilon_A$ the $\epsilon(n,H)>0$ prescribed in Remark 3.3 of \cite{Anderson1990a}, then for $i > N_1(\delta_{\epsilon_A \omega_n, n, H})$ and $r < N_1(\delta_{\epsilon_A \omega_n, n, H})^{-1}$,
\[
\Vol\left(\frac{B^{M_i}_r(p_i)}{r} \right) \geq (1-\epsilon_A) \omega_n.
\]
Thus, for any constant $\eta >0$, there exists a subsequence $(B^{M_{i_j}}_{(1-\eta)r} (p_{i_j}), \frac{1}{r^2}g_{i_j})$  converging in the $C^{1,\alpha}$ topology to $(B^{\RR}_1(0), g_{\RR^n})$. By definition, this means that for $j$ sufficiently large (as prescribed above) there exist a sequence of diffeomorphisms $F_j: B^{\RR^n}_{1-\eta}(0) \to B^{M_{i_j}}_{(1-\eta)r}(p_{i_j})$ such that the metrics $F^*_{j} (\frac{1}{r^2}g_{i_j})$ converge to $g_{\RR^n}$ on $B^{\RR^n}_1(0)$ in the $C^{1, \alpha}$-topology. Now, taking perhaps slightly larger $\eta' > \eta$ and restricting $\left. F_j \right|_{\del B^{\RR^n}_{1-\eta'}(0)}$ to distance spheres, it follows that $(\del B^{M_{i_j}}_{(1-\eta')r}(p_{i_j}), \frac{1}{r^2}g_{i_j})$ also converges to $(\del B^{\RR^n}_{1-\eta'}(0), g_{\RR^n})$ in the $C^{1,\alpha}$ topology. Thus, for $j$ sufficiently large, the two metrics are bounded in the Lipschitz distance as well. 

Focusing now on this subsequence, we employ an estimate of Lakzian-Sormani (see Proposition \ref{lemma-LakzianSormani}) which relates the intrinsic flat distance to the Lipschitz distance of two oriented Riemannian manifolds. Using the notation of Proposition \ref{lemma-LakzianSormani}, take 
\begin{align*}
M_j &= (\del B^{\RR^n}_{1-\eta'}(0), \left. (F_j \right|_{\del B^{\RR^n}_{1-\eta'}(0)})^*\frac{1}{r^2}g_{i_j}),\\
M &= (\del B^{\RR^n}_{1-\eta'}(0), g_{\RR^n}).
\end{align*}
Set, for each $j$, 
\begin{align*}
V_j &= \Vol_{n-1} (M_{i_j}), \hspace{1.1in} a_j  = \frac{\cos^{-1}(1+\frac{1}{j})^{-1}}{\pi} \diam(M_{i_j}),\\
V &= \Vol_{n-1} (M) = n \omega_n (1-\eta')^n, \hspace{.2in} a = \cos^{-1}(1+\frac{1}{j})^{-1} (1-\eta'),
\end{align*}
(note that $\Vol_{n-2}(\del M_j) = \Vol_{n-2}(\del M) = 0$). Also, set $\phi_j : = \max \{a_{i_j}, a\}(V_j + V)$. Note that $\phi_j \searrow 0$ as $j \to \infty$. Summarizing the above arguments and applying Proposition \ref{lemma-LakzianSormani}, we now have the following:

For any $\epsilon>0$, there exists $N_1'(\epsilon, n, H)>0$ such that 
\be\label{F bound by LS}
\dF(M_{i_j}, M) \leq \phi_{j} < \epsilon,
\ee
for all $i > N'_1$ and $r < 1/N'_1$. Note that both $M_{i_j}$ and $M$ are already settled complete. 

To complete the proof, choose an arbitrary $\dbar >0$ and fix some small positive constant $\ebar < \dbar$. It is possible to choose $N_2(\ebar, n)>0$ such that 
\be\label{phi bound}
\phi_j  < \frac{\ebar}{\dbar^n(1+\dbar)},
\ee
for all $j > N_2(\ebar, n)$. Now, take $\overline{N}(\ebar, n, H) = \max\{N_1', N_2\}$. It follows from (\ref{F bound by LS}) and (\ref{phi bound}) that for all $j > \overline{N}$ and $r \in (0, \overline{N}^{-1})$,
\be\label{final dF bound}
\dF\left(\frac{\del B^{M_{i_j}}_{(1-\eta')r}(p_{i_j})}{r}, \del B^{\RR^n}_{1-\eta'}(0)\right) \leq \phi_j < \frac{\ebar}{\dbar^n(1+\dbar)}.
\ee
Therefore, taking  $r \in (0, \overline{N}^{-1}) \cap (0, \dbar)$, any $\eta' \in (0,1)$, and $j$ sufficiently large, it follows (using Lemma  \ref{scalingLemma}, $r < \dbar$, and (\ref{final dF bound}) in turn) that
\begin{align}
\dF(S^{M_{i_j}}_{(1-\eta')r}(p_{i_j}), \del B^{\RR^n}_{(1-\eta')r}(0)) &\leq \dF\left(\frac{S^{M_{i_j}}_{(1-\eta')r}(p_{i_j})}{r}, \del B^{\RR^n}_{1-\eta'}(0) \right)r^n(1+r),\\ 
&\leq \dF\left(\frac{S^{M_{i_j}}_{(1-\eta')r}(p_{i_j})}{r}, \del B^{\RR^n}_{1-\eta'}(0) \right)\dbar^n(1+\dbar),\\
& < \ebar.
\end{align}

Finally, note that there exists $r_{\ebar}>0$ such that 
\be\label{dF lower}
\dF(\del B^{\RR^n}_{r_{\ebar}}(0), \0) = 2\ebar.
\ee
For example, this is achieved by taking $r_{\ebar} = \left(\frac{2\dbar}{\omega_n} \right)^{1/n}$. Thus, for all $r \in (0, \overline{N}^{-1}) \cap (0, \delta)$, and $j >\overline{N}$, one obtains
\[
\dF(S^{M_{i_j}}_{(1-\eta')r}(p_{i_j}), B^{\RR^n}_r(0)) < \ebar = \frac{\dF(B^{\RR^n}_{r_{\ebar}}(0), \0)}{2}.
\]
Therefore, by the triangle inequality for $\dF$ (and the remarks following the statement of Lemma \ref{SormaniPreprintLemma}) it follows that, given any small $\ebar>0$, there exists $\overline{N}(\ebar, n,H)>0$ such that, for any $\eta'\in(0,1)$,
\[
\dF(S^{M_{i_j}}_{(1-\eta')r}(p_{i_j}), \0) > \frac{\dF(B^{\RR^n}_r(0),\0)}{2} = \ebar >0,
\]
for all $r \in (0,\overline{N}^{-1})$ and $j > \overline{N}$.
This provides our contradiction. Namely, given any $\dbar >0$, the measure of the set of $r \in (0, \dbar)$ for which $S^{M_{j}}_{r}(p_{j})$ does not converge to the flat integral current space as $j \to \infty$ is bounded below by $\min\{ \overline{N}^{-1}, \dbar\}>0$. 

{\bf Case 2:} $p_{\infty} \notin \mathcal{R}$.\\
Suppose $p_{\infty}$ is not contained in the regular set of the limit space, so $p_{\infty} \in \mathcal{S}$. By Theorem \ref{S meas}, $\mathcal{R}$ is dense in $X_{\infty}$;  therefore, $p_{\infty} \in  \overline{\mathcal{R}}$, the metric completion of $\mathcal{R}$. Thus, we can find a sequence $(p_{{\infty}})_k \in \R$ so that $(p_{{\infty}})_k \longrightarrow p_{\infty}$ as $k \to \infty$. By Case 1, it follows that $(p_{\infty})_k \in X_{\infty}$, for all $k$, and therefore, $p_{\infty} \in \overline{Y_{\infty}}$. 

Thus, the Claim is proven. Therefore, $X_{\infty} \subset \overline{Y_{\infty}}$ and the theorem follows.
\end{proof}


\bibliographystyle{plain}
\bibliography{}

\end{document}